\documentclass[a4paper,draft]{amsproc}
\usepackage{amssymb}
\usepackage{amscd} 

\theoremstyle{plain}
 \newtheorem{thm}{Theorem}[section]

 \newtheorem{cor}{Corollary}[section]
\theoremstyle{definition}

\theoremstyle{remark}
 \newtheorem{rem}{Remark}[section] 
 \numberwithin{equation}{section}

\renewcommand{\leq}{\leqslant}
\renewcommand{\geq}{\geqslant}
\renewcommand{\setminus}{\smallsetminus}
\title[Lagrange-type operators associated with $U_n^{\varrho}$]{Lagrange-type operators associated with $U_n^{\varrho}$}



\author[H. Gonska]{\bfseries Heiner Gonska}

\address{%
University of Duisburg-Essen\\
Faculty of Mathematics\\
Forsthausweg 2\\
47057 Duisburg\\
Germany}

\email{heiner.gonska@uni-due.de}
\author[I. Ra\c sa]{Ioan Ra\c sa}

\address{Technical University of Cluj-Napoca\\
Department of Mathematics\\
Str. Memorandumului nr. 28\\
RO-400114 Cluj-Napoca\\
Romania}
\email{Ioan.Rasa@math.utcluj.ro}
\author[E.D. St\u anil\u a]{Elena-Dorina St\u anil\u a}
\address{%
University of Duisburg-Essen\\
Faculty of Mathematics\\
Forsthausweg 2\\
47057 Duisburg\\
Germany}
\email{elena.stanila@stud.uni-due.de}

\thanks{Dedicated to Professor Giuseppe Mastroianni on the occasion of his 75th birthday }

\begin{document}

\vspace{18mm}
\setcounter{page}{1}
\thispagestyle{empty}

\begin{abstract}
We consider a class of positive linear operators which, among others, constitute a link between the classical Bernstein operators and the genuine Bernstein-Durrmeyer mappings. The focus is on their relation to certain Lagrange-type interpolators associated to them, a well known feature in the theory of Bernstein operators. Considerations concerning iterated Boolean sums and the derivatives of the operator images are included. Our main tool is the eigenstructure of the members of the class.
\end{abstract}

\maketitle

\section{Introduction}  

The present note continues the authors' research on a class $U_n^{\varrho}$ of Bernstein-type operators which constitute a link between the classical Bernstein operators $B_n (\varrho\rightarrow\infty)$ and the so-called genuine Bernstein-Durrmeyer operators $U_n (\varrho=1).$
This class of operators was introduced by P\u alt\u anea in \cite{Paltanea:2007} and has since then been the subject of several papers dealing with various aspects of the matter.\\
A detailed description is given in the next section.\\
Here we present their relationship to Lagrange interpolation using the eigenstructure of the $U_n^{\varrho}$, thus extending in a natural way results known for $B_n$. The eigenstructure is also useful to describe the convergence behavior of iterated Boolean sums based on a single mapping 
$U_n^{\varrho}$, $\varrho$ and $n$ fixed. In the final Section \ref{sec_derivatives} a relationship between certain divided differences used in Section \ref{sec_dd} and the representation of the derivatives $(U_n^{\varrho})^{(j)}$ is established.

\section{Lagrange type operators associated with $U_n^{\varrho}$} \label{sec_dd}
\subsection{A first description of $L_n^{\varrho}$}
Let $\varrho>0$ and $n\geq 1$ be fixed. Consider the functionals $F_{n,k}^{\varrho}:C[0,1]\rightarrow\mathbb{R}, k=0,1,...,n$, defined by
$$
\begin{array}{l}
F_{n,0}^{\varrho}(f)=f(0), F_{n,n-1}^{\varrho}(f)=f(1)\vspace{0.15cm},\\
F_{n,k}^{\varrho}(f)=\int\limits_{0}^1\dfrac{t^{k\varrho-1}(1-t)^{(n-k)\varrho-1}}{B(k\varrho,(n-k)\varrho)}f(t)dt, \; k=1,...,n-1.
\end{array}
$$
The operator $U_n^{\varrho}:C[0,1]\rightarrow \Pi_{n}$ is given by
$$
U_n^{\varrho}(f;x):=\sum\limits_{k=0}^n F_{n,k}^{\varrho}(f)p_{n,k}(x), f\in C[0,1],
$$
where $p_{n,k}(x)= {n\choose k}x^k(1-x)^{n-k}, x\in[0,1]$.
With a slight abuse of notation consider also the operator $U_n^{\varrho}:\Pi_{n}\rightarrow \Pi_{n}$. Its eigenvalues $\lambda_{n,k}^{(n)}$ and eigenfunctions $p_{n,k}^{(n)}, \; k=0,1,...,n$, are described in \cite{Gonska-Rasa-Stanila:2013}; in particular,
$$1=\lambda_{\varrho,0}^{(n)}=\lambda_{\varrho,1}^{(n)}> \lambda_{\varrho,2}^{(n)}> \lambda_{\varrho,3}^{(n)}>... >\lambda_{\varrho,n}^{(n)}>0,$$
which means that $U_n^{\varrho}:\Pi_{n}\rightarrow \Pi_{n}$ is invertible. Consider the inverse operator $(U_n^{\varrho})^{-1}:\Pi_{n}\rightarrow \Pi_{n}$ and define $L_n^{\varrho}:C[0,1]\rightarrow \Pi_{n}$ by
\begin{equation}\label{lagrange_eq1.1}
L_n^{\varrho}=(U_n^{\varrho})^{-1}\circ U_n^{\varrho}.
\end{equation}
Then $U_n^{\varrho}(L_n^{\varrho}f)=U_n^{\varrho}(f), f\in C[0,1]$, which leads to
\begin{equation}\label{lagrange_eq1.2}
F_{n,k}^{\varrho}(L_n^{\varrho}f)=F_{n,k}^{\varrho}(f), f\in C[0,1], k=0,1,...,n.
\end{equation}
(\ref{lagrange_eq1.2}) expresses an interpolatory property with respect to the functionals \\ $F_{n,0}^{\varrho},...,F_{n,n}^{\varrho}$; more precisely, given $f\in C[0,1], L_n^{\varrho}f$ is the unique polynomial in $\Pi_n$ satisfying (\ref{lagrange_eq1.2}). In particular, $L_n p=p, \forall p\in \Pi_n$. It is known (see \cite{Gonska-Paltanea:2010-1}) that
\begin{equation}\label{lagrange_eq1.3}
\lim\limits_{\varrho\rightarrow\infty}F_{n,k}^{\varrho}(f)=f\left(\frac{k}{n}\right), f\in C[0,1], k=0,1,...,n.
\end{equation}
This entails
\begin{equation}\label{lagrange_eq1.4}
\lim\limits_{\varrho\rightarrow\infty}U_{n}^{\varrho}(f)=B_nf, \;\mbox{uniformly on}\; [0,1],
\end{equation}
for all $f\in C[0,1]$; here $B_n$ denotes the classical Bernstein operator on $C[0,1]$. Let $L_n$ be the Lagrange operator on $C[0,1]$ based on the nodes $0,\frac{1}{n}, ...,\frac{n-1}{n}, 1$. It is easy to see that
\begin{equation}\label{lagrange_eq1.41}
L_n=B_n^{-1}\circ B_n,
\end{equation}
where
$$
C[0,1]\xrightarrow{\;\;\;B_n\;\;\;\;}\Pi_n\xrightarrow{(B_n)^{-1}}\Pi_n.
$$
We will see that
\begin{equation}\label{lagrange_eq1.5}
\lim\limits_{\varrho\rightarrow\infty}L_{n}^{\varrho}(f)=L_nf, \;\mbox{uniformly on}\; [0,1],
\end{equation}
for all $f\in C[0,1]$. If we interpret (\ref{lagrange_eq1.4}) by saying that $U_n^{\infty}=B_n$, then  (\ref{lagrange_eq1.5}) can be interpreted as $L_n^{\infty}=L_n$. On the other hand, one has (see \cite[(3.3)]{Gonska-Rasa-Stanila:2013})
\begin{equation}\label{lagrange_eq1.6}
U_{n}^{\varrho}f=\sum\limits_{k=0}^n \lambda_{\varrho,k}^{(n)}p_{\varrho,k}^{(n)}\mu_{\varrho,k}^{(n)}(f), f\in C[0,1],
\end{equation}
 where $\mu_{\varrho,k}^{(n)}$ are the dual functionals of $p_{\varrho,k}^{(n)}$. This leads to
 \begin{equation*}
L_{n}^{\varrho}(f)=(U_n^{\varrho})^{-1}( U_n^{\varrho}f)=\sum\limits_{k=0}^n \lambda_{\varrho,k}^{(n)}\dfrac{1}{\lambda_{\varrho,k}^{(n)}}p_{\varrho,k}^{(n)}\mu_{\varrho,k}^{(n)}(f),
\end{equation*}
i.e.,
 \begin{equation}\label{lagrange_eq1.7}
L_{n}^{\varrho}(f)=\sum\limits_{k=0}^n p_{\varrho,k}^{(n)}\mu_{\varrho,k}^{(n)}(f), f\in C[0,1].
\end{equation}
So the relationship between $U_{n}^{\varrho}$ and $L_{n}^{\varrho}$, expressed by (\ref{lagrange_eq1.6}) and (\ref{lagrange_eq1.7}), is similar to the relationship between $B_n=U_n^{\infty}$ and $L_n=L_n^{\infty}$, described in \cite[Sect. 6]{Cooper-Waldron:2000}. \\
To conclude this section let us recall that
$$
U_n^{\varrho}=B_n\circ\overline{\mathbb{B}}_{n\varrho}
$$
where
$$\overline{\mathbb{B}}_{r}(f;x)=
\begin{cases}
f(0), x=0;\\
\dfrac{\int\limits_0^1t^{rx-1}(1-t)^{r-rx-1}f(t)dt}{B(rx,r-rx)}, 0<x<1;\\
f(1), x=1.
\end{cases}$$
is the Lupa\c s-M\" uhlbach Beta operator (see \cite[p. 63]{Lupas:1972}, \cite{Muhlbach:1972}). From (\ref{lagrange_eq1.1}) and (\ref{lagrange_eq1.41}) it follows that
\begin{equation}\label{lagrange_eq1.8}
L_{n}^{\varrho}=(\overline{\mathbb{B}}_{n\varrho})^{-1}\circ L_n\circ \overline{\mathbb{B}}_{n\varrho},\;\varrho>0,
\end{equation}
i.e., the operators $L_{n}^{\varrho}$ and $L_n$ are similar.
\subsection{A concrete approach to $L_n^{\varrho}$}
In order to obtain other representations of the operators $L_n^{\varrho}$ we shall use a classical method described, for example, in \cite[Sect. 1.2]{Popoviciu:1973}, \cite{Davis:1975}, \cite[Sect. 1.3]{Mastroianni-Milovanovic:2008}.
Let $n\geq 1, \varrho>0$ and $f\in C[0,1]$ be fixed. Then $L_n^{\varrho}f\in\Pi_n$ has the form $L_n^{\varrho}f=c_0e_0+c_1e_1+...+c_ne_n$, where $e_j(x)=x^j, x\in [0,1], j\geq 0$, and $c_j\in\mathbb{R}$. According to (\ref{lagrange_eq1.2}), the coefficients $c_0,...,c_n$ satisfy the system of equations
$$
\begin{cases}
L_n^{\varrho}f=c_0e_0+c_1e_1+...+c_ne_n\\
F_{n,0}^{\varrho}(f)=c_0F_{n,0}^{\varrho}(e_0)+c_1F_{n,0}^{\varrho}(e_1)+...+c_nF_{n,0}^{\varrho}(e_n)\\
...\\
F_{n,n}^{\varrho}(f)=c_0F_{n,n}^{\varrho}(e_0)+c_1F_{n,n}^{\varrho}(e_1)+...+c_nF_{n,n}^{\varrho}(e_n).
\end{cases}
$$
By eliminating $c_0,...,c_n$, we get
\begin{equation}\label{lagrange_eq2.1}
\left|
\begin{array}{ccccc}
L_n^{\varrho}f& e_0 &e_1 &...& e_n\\
F_{n,0}^{\varrho}(f)& F_{n,0}^{\varrho}(e_0)&F_{n,0}^{\varrho}(e_1)&...&F_{n,0}^{\varrho}(e_n)\\
...&....&...&...&...\\
F_{n,n}^{\varrho}(f)& F_{n,n}^{\varrho}(e_0)&F_{n,n}^{\varrho}(e_1)&...&F_{n,n}^{\varrho}(e_n)
\end{array}
\right|=0
\end{equation}
Since $F_{n,j}^{\varrho}(e_m)=\dfrac{(i\varrho)^{\overline{m}}}{(n\varrho)^{\overline{m}}}$,  where by $x^{\overline{k}}=x(x+1)\cdot...\cdot(x+k-1)$ we have denoted  the
rising factorial,
from (\ref{lagrange_eq2.1}) we get after elementary computations:
\begin{equation}\label{lagrange_eq2.2}
L_n^{\varrho}f=-V(0,\frac{1}{n},...,\frac{n-1}{n},1)^{-1}\left|
\begin{array}{ccccc}
0& e_0 &\dfrac{(n\varrho)^{\overline{1}}}{(n\varrho)}e_1 &...&\dfrac{(n\varrho)^{\overline{n}}}{(n\varrho)^{n}} e_n\\
F_{n,0}^{\varrho}(f)& 1&\dfrac{(0\varrho)^{\overline{1}}}{(n\varrho)}&...&\dfrac{(0\varrho)^{\overline{n}}}{(n\varrho)^{n}}\\
F_{n,1}^{\varrho}(f)& 1&\dfrac{(1\varrho)^{\overline{1}}}{(n\varrho)}&...&\dfrac{(1\varrho)^{\overline{n}}}{(n\varrho)^{n}}\\
...&....&...&...&...\\
F_{n,n}^{\varrho}(f)& 1&\dfrac{(n\varrho)^{\overline{1}}}{(n\varrho)}&...&\dfrac{(n\varrho)^{\overline{n}}}{(n\varrho)^{n}}\\
\end{array}
\right|
\end{equation}
where $V$ is the Vandermonde determinant. Now we are in the position to prove (\ref{lagrange_eq1.5})
\begin{thm}
For each $f\in C[0,1]$ we have
$$
\lim\limits_{\varrho\rightarrow\infty}L_{n}^{\varrho}f=L_n f, \;\mbox{uniformly on}\; [0,1].
$$
\end{thm}
\begin{proof}
Let us remark that
\begin{equation}\label{lagrange_eq2.3}
\lim\limits_{\varrho\rightarrow\infty}\dfrac{(j\varrho)^{\overline{k}}}{(n\varrho)^{k}}=\left(\dfrac{j}{n}\right)^k.
\end{equation}
From (\ref{lagrange_eq1.3}), (\ref{lagrange_eq2.2}), and (\ref{lagrange_eq2.3}) we deduce
\begin{equation}\label{lagrange_eq2.4}
\lim\limits_{\varrho\rightarrow\infty}L_n^{\varrho}f=-V\left(0,\frac{1}{n},...,\frac{n-1}{n},1\right)^{-1}\left|
\begin{array}{ccccc}
0& e_0 &e_1 &...& e_n\\
f(0)& 1&0&...&0\\
f(\frac{1}{n})& 1&\frac{1}{n}&...&(\frac{1}{n})^n\\
...&....&...&...&...\\
f(\frac{n-1}{n})& 1&\frac{n-1}{n}&...&(\frac{n-1}{n})^n\\
f(1)&1&1&...&1
\end{array}
\right|.
\end{equation}
Since the right hand-side of (\ref{lagrange_eq2.4}) is $L_{n}f$ (see, e.g., \cite[Sect. 3.1]{Stancu-Agratini-Coman-Trambitas:2001},  \cite[Sect. 1.3]{Mastroianni-Milovanovic:2008}.), the proof is complete.
\end{proof}
\subsection{The associated divided difference}
The coefficient of $e_n$ in the expression of $L_nf$ is the divided difference of $f$ at the nodes $0,\frac{1}{n}, ....,\frac{n-1}{n},1$, and is given by (see e.g. \cite[Sect. 2.6]{Stancu-Agratini-Coman-Trambitas:2001}):
\begin{eqnarray}\label{lagrange_eq3.1}
\left[0,\frac{1}{n},...,\frac{n-1}{n},1;f\right]=V\left(0,\frac{1}{n},...,\frac{n-1}{n},1\right)^{-1}\times\;\;\;\;\;\;\;\;\;\;\;\;\\
\;\;\;\;\;\;\;\;\;\;\;\;\;\;\;\;\;\;\;\;\;\times\left|
\begin{array}{cccccc}
1& 0 &0 &...& 0&f(0)\\
1&\frac{1}{n}&(\frac{1}{n})^2&...&(\frac{1}{n})^{n-1}&f(\frac{1}{n})\\
...&....&...&...&...&...\\
1&\frac{n-1}{n}&(\frac{n-1}{n})^2&...&(\frac{n-1}{n})^{n-1}&f(\frac{n-1}{n})\\
1&1&1&...&1&f(1)
\end{array}
\right|.\nonumber
\end{eqnarray}
Let us denote by $[F_{n,0}^{\varrho},F_{n,1}^{\varrho},...,F_{n,n}^{\varrho};f]$ the coefficient of $e_n$ in $L_n^{\varrho}f$.
\begin{thm}
For each $f\in C[0,1]$ we have
\begin{eqnarray}\label{lagrange_eq3.2}
[F_{n,0}^{\varrho},F_{n,1}^{\varrho},...,F_{n,n}^{\varrho};f]=\dfrac{(n\varrho)^{\overline{n}}}{(n\varrho)^n}V\left(0,\frac{1}{n},...,\frac{n-1}{n},1\right)^{-1} \times\;\;\;\;\;\;\;\;\;\;\;\;\\
\;\;\;\;\;\;\;\;\;\;\;\;\;\;\;\;\;\;\;\;\;\times\left|
\begin{array}{cccccc}
1& 0 &0 &...& 0&F_{n,0}^{\varrho}(f)\\
1&\frac{1}{n}&(\frac{1}{n})^2&...&(\frac{1}{n})^{n-1}&F_{n,1}^{\varrho}(f)\\
...&....&...&...&...&...\\
1&\frac{n-1}{n}&(\frac{n-1}{n})^2&...&(\frac{n-1}{n})^{n-1}&F_{n,n-1}^{\varrho}(f)\\
1&1&1&...&1&F_{n,n}^{\varrho}(f)
\end{array}\nonumber
\right|.
\end{eqnarray}
\end{thm}
\begin{proof}
From (\ref{lagrange_eq2.2}) we get immediately
\begin{eqnarray*}
\begin{array}{l}
[F_{n,0}^{\varrho},F_{n,1}^{\varrho},...,F_{n,n}^{\varrho};f]=\dfrac{(n\varrho)^{\overline{n}}}{(n\varrho)^n}V\left(0,\frac{1}{n},...,\frac{n-1}{n},1\right)^{-1}\times\;\;\;\;\;\;\;\;\;\;\;\;\\
\;\;\;\;\;\;\;\;\;\;\;\;\;\;\;\;\;\;\;\;\;\times\left|
\begin{array}{ccccc}
1& \dfrac{(0\varrho)^{\overline{1}}}{n\varrho} &...& \dfrac{(0\varrho)^{\overline{n-1}}}{(n\varrho)^{n-1}} &F_{n,0}^{\varrho}(f)\\
1& \dfrac{(1\varrho)^{\overline{1}}}{n\varrho} &...&\dfrac{(1\varrho)^{\overline{n-1}}}{(n\varrho)^{n-1}}&F_{n,1}^{\varrho}(f)\\
...&...&...&...&...\\
1& \dfrac{(n\varrho)^{\overline{1}}}{n\varrho} &...&\dfrac{(n\varrho)^{\overline{n-1}}}{(n\varrho)^{n-1}}&F_{n,n}^{\varrho}(f)
\end{array}
\right|=\vspace{0.3cm}\\
=\dfrac{(n\varrho)^{\overline{n}}/(n\varrho)^n}{(n\varrho)^{\frac{n(n-1)}{2}}V\left(0,\frac{1}{n},...,\frac{n-1}{n},1\right)}\left|
\begin{array}{ccccc}
1& (0\varrho)^{\overline{1}} &...& (0\varrho)^{\overline{n-1}} &F_{n,0}^{\varrho}(f)\\
1&(1\varrho)^{\overline{1}}&...&(1\varrho)^{\overline{n-1}}&F_{n,1}^{\varrho}(f)\\
...&...&...&...&...\\
1& (n\varrho)^{\overline{1}} &...&(n\varrho)^{\overline{n-1}}&F_{n,n}^{\varrho}(f)
\end{array}
\right|=
\vspace{0.3cm}\\
=\dfrac{(n\varrho)^{\overline{n}}/(n\varrho)^n}{(n\varrho)^{\frac{n(n-1)}{2}}V\left(0,\frac{1}{n},...,\frac{n-1}{n},1\right)}\left|
\begin{array}{ccccc}
1& 0&...& 0 &F_{n,0}^{\varrho}(f)\\
1&\varrho&...&\varrho^{n-1}&F_{n,1}^{\varrho}(f)\\
...&...&...&...&...\\
1& n\varrho &...&(n\varrho)^{n-1}&F_{n,n}^{\varrho}(f)
\end{array}
\right|
\end{array}
\end{eqnarray*}
and this leads to (\ref{lagrange_eq3.2}).
\end{proof}
\begin{rem}\label{Lagrange_rmk1}
From (\ref{lagrange_eq1.3}),  (\ref{lagrange_eq3.1}) and  (\ref{lagrange_eq3.2}) we derive
$$
\lim\limits_{\varrho\rightarrow\infty}[F_{n,0}^{\varrho},F_{n,1}^{\varrho},...,F_{n,n}^{\varrho};f]=\left[0,\frac{1}{n},...,\frac{n-1}{n},1;f\right]
$$
for all $f\in C[0,1]$.
Moreover, let $f\in C[0,1]$ and $\Phi_n$ a (Lagrange type) polynomial
with $\Phi_n\in\Pi_n, \Phi_n(\frac{j}{n})=F_{n,j}^{\varrho}(f), j=0,...,n$. From (\ref{lagrange_eq3.1}) and (\ref{lagrange_eq3.2}) it is easy to deduce
\begin{equation*}\label{lagrange_eq3.3}
[F_{n,0}^{\varrho},F_{n,1}^{\varrho},...,F_{n,n}^{\varrho};f]=\dfrac{(n\varrho)^{\overline{n}}}{(n\varrho)^{n}}\left[0,\frac{1}{n},...,\frac{n-1}{n},1;\Phi_n\right].
\end{equation*}
The last (classical) divided difference can be computed by recurrence; see \cite{Stancu-Agratini-Coman-Trambitas:2001}.
\end{rem}
\begin{rem}
Using (\ref{lagrange_eq1.7}), we see that
\begin{equation}\label{lagrange_eq3.4}
\mu_{\varrho,n}^{(n)}=[F_{n,0}^{\varrho},F_{n,1}^{\varrho},...,F_{n,n}^{\varrho};\cdot].
\end{equation}
For the Bernstein operator (i.e., for $\varrho\rightarrow\infty$), (\ref{lagrange_eq3.4}) can be found in \cite[p.164]{Cooper-Waldron:2000} .
\end{rem}

\begin{rem}\label{lagrange_rmk3}
$u_{n+1}^{\varrho}:=e_{n+1}-L_n^{\varrho}e_{n+1}$ is the unique monic polynomial in $\Pi_{n+1}$ such that $L_n^{\varrho}u_{n+1}^{\varrho}=0$. For example, $u_{n+1}^{\infty}=x(x-1)(x-\frac{1}{n})\cdot...\cdot(x-\frac{n-1}{n})$. Moreover, $u_{n+1}^{1}(x)=x(x-1)J_{n-1}(x)$, where $J_0(x), J_1(x),...$ are the monic Jacobi polynomials, orthogonal on $[0,1]$ with respect to the weight function $x(1-x)$. Indeed, $F_{n,0}^{1}(u_{n+1}^1)=F_{n,n}^{1}(u_{n+1}^1)=0$, and $\int\limits_0^1t^{k-1}(1-t)^{n-k-1}u_{n+1}^1(t)dt=\int\limits_0^1t^{k-1}(1-t)^{n-k-1}t(t-1)J_{n-1}(t)dt=0$ (since for all $k=1,...,n-1$, $t^{k-1}(1-t)^{n-k-1}$ is a polynomial of degree $n-2$). This implies $F_{n,k}^{1}(u_{n+1}^1)=0, k=1,...,n-1,$ and so $L_n^1(u_{n+1}^1)=0$.
\end{rem}
Now we shall prove a general result.
\begin{thm}\label{lagrange_th3}
The polynomial $u_{n+1}^{\varrho}$ has $n+1$ distinct roots in $[0,1]$.
\end{thm}
\begin{proof}
By using Remark \ref{lagrange_rmk3} and (\ref{lagrange_eq1.8}) we get $(\overline{\mathbb{B}}_{n\varrho}^{-1}\circ L_n\circ \overline{\mathbb{B}}_{n\varrho})(u_{n+1}^{\varrho})=0$,
which entails $L_n(\overline{\mathbb{B}}_{n\varrho}u_{n+1}^{\varrho})=0$. Now the same Remark \ref{lagrange_rmk3} yields
$$
\overline{\mathbb{B}}_{n\varrho}u_{n+1}^{\varrho}=\dfrac{(n\varrho)^{n+1}}{(n\varrho)^{\overline{n+1}}}u_{n+1}^{\infty}.
$$
So $\overline{\mathbb{B}}_{n\varrho}u_{n+1}^{\varrho}$ has $n+1$ distinct roots in $[0,1]$. According to \cite{Kacso-Stanila:2013}, $u_{n+1}^{\varrho}$ has at least $n+1$ distinct roots in $[0,1]$; to finish the proof, it suffices to remark that $u_{n+1}^{\varrho}$ is a polynomial of degree $n+1$.
\end{proof}
Now let us recall the representation of $L_n$ in terms of the fundamental Lagrange polynomials:
$$
L_nf(x)=\sum\limits_{k=0}^{n}l_{n,k}(x)f(\frac{k}{n}), f\in C[0,1], x\in [0,1].
$$
Using (\ref{lagrange_eq1.8}) we infer that $L_n^{\varrho}$ has a similar representation, namely
$$
L_n^{\varrho}f(x)=\sum\limits_{k=0}^{n}l_{n,k}^{\varrho}(x)F_{n,k}^{\varrho}(f),
$$
where
\begin{equation}\label{lagrange_eq3.5}
l_{n,k}^{\varrho}:=\overline{\mathbb{B}}_{n\varrho}^{-1}(l_{n,k}), k=0,1,...,n.
\end{equation}
\begin{thm}
For each $k=0,1,...,n$, the polynomial $l_{n,k}^{\varrho}$ has $n$ distinct roots in [0,1].
\end{thm}
\begin{proof}
Since, according to (\ref{lagrange_eq3.5}), $\overline{\mathbb{B}}_{n\varrho}(l_{n,k}^{\varrho})=l_{n,k}$, the proof is similar to that of Theorem \ref{lagrange_th3} and we omit it.
\end{proof}
In what follows we shall establish mean value theorems for the generalized divided difference and for the remainder $R_n^{\varrho}f:=f-L_n^{\varrho}$.
\begin{thm}\label{lagrange_th3.4}
Let $n\geq 1, \varrho>0$ and $f\in C[0,1]$ be given. Then there exist $0=t_0<t_1<....<t_n=1$ such that
\begin{equation}\label{lagrange_eq3.6}
R_n^{\varrho}f(t_i)=0, \; i=0,1,...,n.
\end{equation}
\end{thm}
\begin{proof}
According to (\ref{lagrange_eq1.2}), $F_{n,k}^{\varrho}(R_n^{\varrho}f)=0, k=0,1,...,n$, i.e.
\begin{equation}\label{lagrange_eq3.7}
R_n^{\varrho}f(0)=R_n^{\varrho}f(1)=0,
\end{equation}
\begin{eqnarray}
\int\limits_0^1t^{k\varrho-1}(1-t)^{(n-k)\varrho-1}R_n^{\varrho}f(t)dt=0, k=1,...,n-1.\label{lagrange_eq3.8}
\end{eqnarray}
Set $x:=\left(\dfrac{t}{1-t}\right)^{\varrho}, \;j:=k-1$, and $h(x):=R_n^{\varrho}f\left(\dfrac{x^{1/\varrho}}{1+x^{1/\varrho}}\right), x\geq 0$. Then (\ref{lagrange_eq3.8}) becomes
\begin{equation}\label{lagrange_eq3.9}
\int\limits_0^{\infty} (1+x^{1/\varrho})^{-n\varrho}x^j h(x)dx=0, \; j=0,1,...,n-2.
\end{equation}
Suppose that the number of the roots of $h$ in $(0,+\infty)$ is at most $n-2$, i.e. $\{x\in (0,+\infty):h(x)=0\}=\{x_1,...,x_r\}, r\leq n-2$. Then there exists a polynomial $p\in \Pi_{n-2}$ such that
$\{x\in (0,+\infty):p(x)=0\}\subset\{x_1,...,x_r\}$ and, moreover,
\begin{equation}\label{lagrange_eq3.10}
\int\limits_0^{\infty} (1+x^{1/\varrho})^{-n\varrho}p(x) h(x)dx>0.
\end{equation}
Obviously (\ref{lagrange_eq3.10}) contradicts (\ref{lagrange_eq3.9}), which means that $h$ has at least $n-1$ roots in $ (0,+\infty)$. It follows that $R_n^{\varrho}f$ has at least $n-1$ roots in $(0,1)$. Together with  (\ref{lagrange_eq3.7}), this proves the theorem.
\end{proof}
\begin{cor}\label{lagrange_cor3.5}
Let $n\geq 1, \varrho>0$ and $f\in C^n[0,1]$ be given. Then there exists $\xi\in(0,1)$ such that
$$
[F_{n,0}^{\varrho},F_{n,1}^{\varrho},...,F_{n,n}^{\varrho};f]=\dfrac{f^{(n)}(\xi)}{n!}.
$$
\end{cor}
\begin{proof}
According to Theorem \ref{lagrange_th3.4}, $R_n^{\varrho}f$ has at least $n+1$ roots in $[0,1].$ It follows that $(R_n^{\varrho}f)^{(n)}$ has at least a root $\xi\in(0,1)$. Thus
$$
0=(R_n^{\varrho}f)^{(n)}(\xi)=f^{(n)}(\xi)-n![F_{n,0}^{\varrho},F_{n,1}^{\varrho},...,F_{n,n}^{\varrho};f],
$$
and the proof is finished.
\end{proof}
Let now $n\geq 1, \varrho>0$ and $f\in C^{n+1}[0,1]$ be given. Consider the points $t_0,t_1,...,t_n$ satisfying (\ref{lagrange_eq3.6}), and let $\omega(t)=(t-t_0)\cdot...\cdot(t-t_n)$.
\begin{cor}\label{lagrange_cor3.6}
Let $x\in [0,1]\setminus\{t_0,t_1,...,t_n\}$. Under the above assumption there exists $\eta_x\in(0,1)$ such that
\begin{equation*}\label{lagrange_eq3.11}
R_n^{\varrho}f(x)=\omega(x)\dfrac{f^{(n+1)}(\eta_x)}{(n+1)!}.
\end{equation*}
\end{cor}
\begin{proof}
Consider the function $w(t)=\omega(x)R_n^{\varrho}f(t)-\omega(t)R_n^{\varrho}f(x), t\in[0,1]$. Then $x,t_0,...,t_n$ are roots of $w$, which means that there exists $\eta_x\in(0,1)$ such that $w^{(n+1)}(\eta_x)=0$. Now it suffices to remark that
$w^{(n+1)}(t)=\omega(x)f^{(n+1)}(t)-(n+1)!R_n^{\varrho}f(x)$.
\end{proof}
Corollaries \ref{lagrange_cor3.5} and \ref{lagrange_cor3.6} generalize the mean value theorems for the divided difference and the remainder in classical Lagrange interpolation; see \cite[Sect. 3.1]{Stancu-Agratini-Coman-Trambitas:2001},  \cite[Sect. 1.4]{Mastroianni-Milovanovic:2008}..
\section{Iterated Boolean sums of the operators $U_n^{\varrho}$}
For $M\geq 1$, let $$
\oplus^{M}U_{n}^{\varrho}=I-(I-U_{n}^{\varrho})^M.
$$ be the iterated Boolean sum of $U_n^{\varrho}$; here $I$ stands for the identity operator on $C[0,1]$. 
Iterated Boolean sums of the classical Bernstein operator and modifications thereof were investigated by numerous authors in the past, among them G. Mastroiani and M.R. Occorsio (see \cite{Mastroianni-Occorsio:1977},\cite{Mastroianni-Occorsio:1978}). Some historical information on this method which may be traced to I.P. Natanson can be found in \cite{Gonska-Zhou:1992}.
From a general result of H.J. Wenz \cite[Theorem 2]{Wenz:1997} it follows that
$\lim\limits_{M\rightarrow\infty}\oplus^{M}U_{n}^{\varrho}f=L_{n}^{\varrho}f, f\in C[0,1], n\geq 1$.
With the notation from the preceding sections, we can say  more, namely
\begin{thm}
Let $n\geq 2$ and $f\in C[0,1]$ be given. Then
\begin{equation}\label{lagrange_eq4.1}
\lim\limits_{M\rightarrow\infty}(1-\lambda_{\varrho,n}^{(n)})^{-M}(\oplus^{M}U_{n}^{\varrho}f-L_{n}^{\varrho}f)=-[F_{n,0}^{\varrho},F_{n,1}^{\varrho},...,F_{n,n}^{\varrho};f]p_{\varrho,n}^{(n)},
\end{equation}
uniformly on $[0,1]$.
\end{thm}
\begin{proof}
We have, according to (\ref{lagrange_eq1.6})
\begin{eqnarray*}
\oplus^{M}U_{n}^{\varrho}f&=&(I-(I-U_{n}^{\varrho})^M)f=\sum\limits_{i=1}^M(-1)^{i+1}{M \choose i}(U_{n}^{\varrho})^{i}f\\
&=&\sum\limits_{i=1}^M(-1)^{i+1}{M \choose i}\sum\limits_{k=0}^n(\lambda_{\varrho,k}^{(n)})^{i}p_{\varrho,k}^{(n)}\mu_{\varrho,k}^{(n)}(f)\\
&=&\sum\limits_{k=0}^np_{\varrho,k}^{(n)}\mu_{\varrho,k}^{(n)}(f)
\sum\limits_{i=1}^M(-1)^{i+1}{M \choose i}(\lambda_{\varrho,k}^{(n)})^{i}\\
&=& \sum\limits_{k=0}^np_{\varrho,k}^{(n)}\mu_{\varrho,k}^{(n)}(f)(1-(1-(\lambda_{\varrho,k}^{(n)})^M).
\end{eqnarray*}
Combined with (\ref{lagrange_eq1.7}) this yields
$$
\oplus^{M}U_{n}^{\varrho}f-L_{n}^{\varrho}f=-\sum\limits_{k=0}^np_{\varrho,k}^{(n)}\mu_{\varrho,k}^{(n)}(f)(1-\lambda_{\varrho,k}^{(n)})^M,
$$
i.e.\\ $(1-\lambda_{\varrho,n}^{(n)})^{-M}(\oplus^{M}U_{n}^{\varrho}f-L_{n}^{\varrho}f)=
-p_{\varrho,n}^{(n)}\mu_{\varrho,n}^{(n)}(f)-\sum\limits_{k=0}^{n-1}\mu_{\varrho,k}^{(n)}(f)\mu_{\varrho,k}^{(n)}(f)\left(\dfrac{1-\lambda_{\varrho,k}^{(n)}}{1-\lambda_{\varrho,n}^{(n)}}\right)^M$.
Since $0<\dfrac{1-\lambda_{\varrho,k}^{(n)}}{1-\lambda_{\varrho,n}^{(n)}}<1, k=0,...,n-1$, we get
$$
\lim\limits_{M\rightarrow\infty}(1-\lambda_{\varrho,n}^{(n)})^{-M}(\oplus^{M}U_{n}^{\varrho}f-L_{n}^{\varrho}f)=-\mu_{\varrho,n}^{(n)}(f)p_{\varrho,n}^{(n)}.
$$
To conclude the proof it suffices to use (\ref{lagrange_eq3.4}).
\end{proof}
\begin{rem}
For $\varrho\rightarrow\infty$, (\ref{lagrange_eq4.1}) was obtained in \cite[Th. 26.7]{Rasa-Vladislav:1999}.
\end{rem}
\section{The derivatives of $U_n^{\varrho}$}\label{sec_derivatives}
In this section we show that there is a natural relationship between the derivatives of the operator images and the divided differences $[...;\Phi_n]$ which we introduced in Remark \ref{Lagrange_rmk1}.
\begin{thm}
With the usual notation the following relationships hold:
\begin{itemize}
\item[(i)] $(U_{n}^{\varrho}(f;x))'=n\sum\limits_{k=0}^{n-1}p_{n-1,k}(x)\Delta^1 F_{n,k}^{\varrho}(f)=\sum\limits_{k=0}^{n-1}p_{n-1,k}(x)\left[\dfrac{k}{n}, \dfrac{k+1}{n};\Phi_n\right]$;
\item[(ii)] $(U_{n}^{\varrho}(f;x))^{(j)}=n(n-1)\cdot...\cdot(n-j+1)\sum\limits_{k=0}^{n-j}p_{n-j,k}(x)\Delta^j F_{n,k}^{\varrho}(f)$\\$=n(n-1)\cdot...\cdot(n-j+1)\sum\limits_{k=0}^{n-j}p_{n-j,k}(x)\dfrac{j!}{n^j}\left[\dfrac{k}{n}, ...,\dfrac{k+j}{n};\Phi_n\right]$;
\item[(iii)] $U_{n}^{\varrho}(f;x)=\sum\limits_{k=0}^{n}{n\choose k}\Delta^k F_{n,0}^{\varrho}(f)e_k(x)=\sum\limits_{k=0}^{n}{n\choose k}\dfrac{k!}{n^k}\left[0,\dfrac{1}{n}, ...,\dfrac{k}{n};\Phi_n\right]e_k(x)$;
\end{itemize}
where as before $\Phi_n\left(\frac{k}{n}\right)=F_{n,k}^{\varrho}(f)$. 
\end{thm}
\begin{proof}
(i) The forward difference was defined in \cite[p. 792]{Gonska-Paltanea:2010-1} by:
$$
\Delta^jF_{n,k}^{\varrho}(f)=\sum\limits_{i=0}^{j}{j\choose i}(-1)^{i+j}
F_{n,k+i}^{\varrho}(f).
$$
Thus we have
$$\left[\dfrac{k}{n}, \dfrac{k+1}{n};\Phi_n\right]=\dfrac{\Phi_n(\frac{k+1}{n})-\Phi_n(\frac{k}{n})}{\frac{k+1}{n}-\frac{k}{n}}=n[F_{n,k+1}^{\varrho}(f)-F_{n,k}^{\varrho}(f)]=n\Delta^1F_{n,k}^{\varrho}(f);$$
(ii) The first equality can be found in  \cite[p. 792]{Gonska-Paltanea:2010-1}. It remains to show that $$\Delta^j F_{n,k}^{\varrho}(f)=\dfrac{j!}{n^j}\left[\dfrac{k}{n}, ...,\dfrac{k+j}{n};\Phi_n\right].$$ 
We have that
$\Delta^{j+1}F_{n,k}^{\varrho}(f)=\Delta(\Delta^{j}F_{n,k}^{\varrho}(f))=
\Delta^{j}F_{n,k+1}^{\varrho}(f)-\Delta^{j}F_{n,k}^{\varrho}(f).$
By using the recurrence formula for divided differences (see e.g. \cite[p.104]{Stancu-Agratini-Coman-Trambitas:2001}) we get:
\begin{eqnarray*}
\Delta^{j}F_{n,k+1}^{\varrho}(f)-\Delta^{j}F_{n,k}^{\varrho}(f)&=&\frac{j!}{n^j}\cdot\frac{j+1}{n}\cdot\frac{\left[\frac{k+1}{n}, ...,\frac{k+j+1}{n};\Phi_n\right]-\left[\frac{k}{n}, ...,\frac{k+j}{n};\Phi_n\right]}{\frac{k+j+1}{n}-\frac{k}{n}}\\
&=&\frac{(j+1)!}{n^{j+1}}\left[\frac{k}{n}, ...,\frac{k+j+1}{n};\Phi_n\right]=\Delta^{j+1}F_{n,k}^{\varrho}(f).
\end{eqnarray*}
(iii) We apply Taylor's formula to $U_n^{\varrho}$ of degree $n$
$$
U_n^{\varrho}(f;x)=\sum\limits_{j=0}^n\dfrac{(U_n^{\varrho}f)^{(j)}(0)}{j!}x^k
$$
and show that $(U_{n}^{\varrho}(f;x))^{(j)}=n(n-1)\cdot...\cdot(n-j+1)\Delta^j F_{n,0}^{\varrho}(f)$. To this end we take $x=0$ in (ii); because $p_{n-j,0}(0)=1$ and for all $k\geq 1, p_{n-j,k}(0)=0$, from $\sum\limits_{k=0}^{n-j}$ only the first term remains, which concludes the proof.  
\end{proof}
\begin{rem}
In the case $\varrho\rightarrow\infty$ we can find the analogues of the above relationships in \cite[p. 300-302]{Stancu-Agratini-Coman-Trambitas:2001}.
\end{rem}

\bibliographystyle{amsplain}

\end{document}